\documentclass[11pt]{amsart}

\usepackage{amsmath}
\usepackage{amsfonts}
\usepackage{amssymb}
\usepackage{latexsym}

\title[Mean field games model of wealth]{Existence theory for a time-dependent mean field games model of household 
wealth}
\author{David M. Ambrose}
\address{3141 Chestnut St., Department of Mathematics, Drexel University, Philadelphia, PA 19104 USA}
\email{dma68@drexel.edu}

\newtheorem{theorem}{Theorem}
\newtheorem{lemma}[theorem]{Lemma}
\newtheorem{remark}[theorem]{Remark}

\begin{document}

\begin{abstract} We study a nonlinear system of 
partial differential equations arising in macroeconomics which utilizes a mean
field approximation.  This system together with the corresponding data, subject to two moment constraints, 
is a model for debt and wealth across a large number of similar households, and
was introduced in a recent paper of Achdou, Buera, Lasry, Lions, and Moll.  We introduce a relaxation of their problem,
generalizing one of the moment constraints; any solution of the original model is a solution of this relaxed problem.
We prove existence and uniqueness of strong solutions to the relaxed problem, under the assumption that the time
horizon is small.  Since these solutions are unique and since solutions of the original problem are also solutions of the
relaxed problem, we conclude that if the original problem does have solutions, then such solutions 
must be the solutions we prove to exist.  Furthermore, for some data and for sufficiently small time horizons,
we are able to show that solutions of the relaxed problem are in fact not solutions of the original problem.
In this way we demonstrate nonexistence of solutions for the original problem in certain cases.
\end{abstract}

\maketitle

\section{Introduction}

A recent paper of Achdou, Buera, Lasry, Lions, and Moll calls attention to PDE models in macroeconomics;
we study a model proposed there for the distribution of wealth across many similar households \cite{mollPhilTrans}.
In this model, the independent variables are $a,$ wealth, $z,$ income, and $t,$ time.  Each household of a given
wealth and income must decide how much of their income to put towards consumption and how much to instead
save.  Note that wealth and savings can be positive or negative, representing debt for negative values.
The authors make a mean field assumption in the modeling, so that a representative household is seen as interacting
not with all the many other individual households, but only with the aggregation of these.  In addition to introducing
the model, the authors of \cite{mollPhilTrans} work with stationary solutions and state that existence and uniqueness of
time-dependent solutions is an open problem.  The present work gives the first theory of existence and uniqueness
for time-dependent solutions.

The particular nonlinear PDE model from \cite{mollPhilTrans} is given by the two equations
\begin{equation}\label{vEquation}
\partial_{t}v+\frac{1}{2}\sigma^{2}(z)\partial_{zz}v+\mu(z)\partial_{z}v+(z+r(t)a)\partial_{a}v
+H(\partial_{a}v)-\rho v=0,
\end{equation}
\begin{equation}\label{gEquation}
\partial_{t}g-\frac{1}{2}\partial_{zz}(\sigma^{2}(z)g)+\partial_{z}(\mu(z)g)
+\partial_{a}((z+r(t)a)g)+\partial_{a}(gH_{p}(\partial_{a}v))=0.
\end{equation}
The dependent variables are $g,$ the distribution of households, and $v,$ the present discounted
value of future utility derived
from consumption; the discount rate is $\rho.$
The nonlinear function $H$ is the Hamiltonian for the problem and is related to a given utility function, $u;$ the
specific form of $H$ is given below in Section \ref{formulationSection}.
We consider the $z$ variable to be taken from the doman $[z_{\mathrm{min}},z_{\mathrm{max}}],$ and the 
$a$ variable to be taken from $\mathbb{R}.$  
The function $\sigma\geq0$ is a diffusion coefficient and the function $\mu$ is a transport coefficient.  We take these
to be smooth and to satisfy $\sigma(z_{min})=\sigma(z_{max})=0$ and $\mu(z_{min})=\mu(z_{max})=0,$ so there
is no transport or diffusion through the boundary of the domain.  The interest rate $r(t)$ is not given but instead depends
on the unknowns; determining $r$ will be a major focus of the present work.
The model is based on models appearing previously in the economics literature 
\cite{aiyagari}, \cite{bewley}, \cite{huggett}.

Our choice of domain with respect to the $a$ variable is a different from \cite{mollPhilTrans}, in which the $a$ variable
was taken from the semi-infinite interval $[a_{\mathrm{min}},\infty)$ for a given value $a_{min}<0.$
The theorem we prove will be for compactly supported distributions $g,$ and thus our theorem is consistent
with \cite{mollPhilTrans} with respect to the spatial domain as long as $a_{min}$ is taken to be beyond the edge of
the support of our $g,$ especially at the initial time.
At the end, in Section \ref{discussionSection}, we will discuss further the 
restriction of our solutions to the domain given in \cite{mollPhilTrans}.

We have two moment conditions which must be satisfied:
\begin{equation}\label{gZeroMoment}
\int g \ dadz =1,
\end{equation}
\begin{equation}\label{gFirstMoment}
\int ag\ dadz =0.
\end{equation}
Of course condition \eqref{gZeroMoment} simply expresses that $g$ is a probability measure.
On the other hand \eqref{gFirstMoment} is an equilibrium condition which expresses that the system is closed
in the sense that all money available to be borrowed in the system is in fact borrowed, and conversely all money
borrowed in the system comes from within the system.  Restated, condition \eqref{gFirstMoment} expresses
that households with negative wealth have borrowed from households with positive wealth, that households with
positive wealth have lent to households with negative wealth, and these total amounts borrowed and lent balance
with each other.  It is from the condition \eqref{gFirstMoment} that the interest rate, $r(t),$ is to be determined.

The equation \eqref{vEquation} for $v$ is backward parabolic, while the equation \eqref{gEquation} for $g$ is
forward parabolic; this is the typical situation for mean field games.  We therefore specify initial data $g_{0}$ for
$g,$ giving an initial distribution of households, and terminal data $v_{T}$ for $v,$ giving a final utility function.

We actually are not able to fully solve the problem specified by \eqref{vEquation},
\eqref{gEquation}, \eqref{gZeroMoment}, \eqref{gFirstMoment}, with the accompanying data; rather than 
being a defect of our method, we are able to prove in some cases that this problem does not have a solution.
In \cite{mollPhilTrans}, the authors did not indicate that a general terminal condition $v_{T}$ should be specified,
but instead indicated a particular choice: that $T$ should be taken to be large and that $v_{T}$ should be 
associated to a stationary solution of the system.  We will discuss this proposed restriction on the data further
in our concluding section, Section \ref{discussionSection} below.

Another condition was stated in \cite{mollPhilTrans}, which is related to their choice of the spatial domain
with respect to the $a$ variable being $[a_{min},\infty).$  Since the equations \eqref{vEquation}, \eqref{gEquation}
include transport terms with respect to $a,$ a boundary condition at $a=a_{min}$ must be carefully given.
This is the ``state constraint boundary condition'' of \cite{mollPhilTrans}, which indicates that the relevant
characteristics point into the domain; such boundary conditions for transport equations have been developed by
Feller \cite{feller57}.  The existence of the boundary at $a_{min}$ is a modeling decision, stating that 
lenders will no longer lend to households with debt of $a_{min};$ the state constraint boundary condition then 
implies that for these households, their incomes are necessarily high enough that in the absence of further borrowing,
their debt load will not increase from the accumulating interest.  By considering compactly supported solutions 
and taking the support to be away from a given value of $a_{min},$ we obviate the need for any such state constraint
boundary condition.  Furthermore, with our compactly supported distribution $g,$ our solutions feature
a maximum and minimum wealth at each time, but these maximum and minimum values are not fixed in time.  

The system \eqref{vEquation}, \eqref{gEquation} is an example from the realm of mean field games, which
have been introduced by Lasry and Lions \cite{lasryLions1}, \cite{lasryLions2}, \cite{lasryLions3}, 
and also by Caines, Huang, and Malhame \cite{huang1}, \cite{huang2}, to study problems in game
theory with a large number of similar agents.  Existence theory for such systems has been developed by several
authors \cite{degenerateMFG}, \cite{cirantSobolev}, \cite{gomesSuper}, \cite{gomesLog}, \cite{gomesSub},
\cite{porretta1}, \cite{porretta2}, \cite{porretta3},
but the system \eqref{vEquation}, \eqref{gEquation} does not fall readily into any previously developed
existence theory for two main reasons.  First, some existence theory such as that of the author relies strongly
on the presence of parabolic effects \cite{ambroseMFG1}, \cite{ambroseMFG2}, \cite{ambroseMFG3}, 
but in \eqref{vEquation}, \eqref{gEquation} the diffusion is anisotropic and
cannot be used to bound derivatives with respect to the $a$ variable.  Second, many of these
works assume structure on the nonlinearity, especially additive separability into a part which depends on $v$
and a part which depends on $g,$ and this separability is not present here.  Instead, the unknowns interact
through the interest rate $r(t),$ and this multiplies other terms in the equations.

The author's prior works \cite{ambroseMFG1}, \cite{ambroseMFG2},
\cite{ambroseMFG3} could be described as viewing the mean field games system as a coupled pair of nonlinear
heat equations.  With the anisotropic effects, we now take the view instead that \eqref{vEquation}, \eqref{gEquation}
form a coupled pair of nonlinear transport equations.  Otherwise, once we have reformulated the system appropriately,
the method used to prove existence and uniqueness of solutions is broadly similar to that of the author's prior
work \cite{ambroseMFG3}; this is the energy method, but adapted to the forward-backward setting of mean field games.

The plan of the paper is as follows: in Section \ref{formulationSection} we make some reformulation of the problem,
changing to a more convenient variable than $v.$  In Section \ref{interestRateSection} we take care to discuss how the
interest rate $r(t)$ is calculated, introducing a modification of the original problem.  In Section \ref{iterativeSection}
we set up an approximation scheme for solving our modified problem.  In Section \ref{existenceAndBounds} we 
prove that our approximate problems have solutions, and develop bounds for the solutions which are uniform in the
approximation parameters.  We pass to the limit to find solutions of our modified problem in Section \ref{limitSection},
to complete our existence proof.  We then prove uniqueness of these solutions in Section \ref{uniquenessSection}.
Finally, we make some concluding remarks in Section \ref{discussionSection}, 
including pointing out that our existence theory for the modified problem
demonstrates that the original problem in some cases in fact has no solution.  Our main theorems are
Theorem \ref{mainExistenceTheorem} in Section \ref{limitSection}, which establishes existence of solutions 
to our modified problem, and Theorem \ref{mainUniquenessTheorem} in Section \ref{uniquenessSection}, 
which establishes uniqueness of these solutions.

\section{Formulation}\label{formulationSection}

We have the Hamiltonian satisfying 
\begin{equation}\nonumber
H(p)=\max_{c\geq0}\left(-cp+u(c)\right),
\end{equation}
where $u$ is a given consumer utility function.  Since $u$ is a consumer utility function, standard economic 
assumptions are
that $u'(c)>0$ for all $c$ and $u''(c)<0$ for all c.  For simplicity, we take $u$ to be infinitely smooth away from $c=0,$
and we also assume for simplicity that the range of $u'$ is 
$(0,\infty)$ and thus the domain of $(u')^{-1}$ is also $(0,\infty).$   We will comment briefly on the general case, 
in our concluding remarks in Section \ref{discussionSection}.

Doing some calculus we see that $-cp+u(c)$ is maximized when $p=u'(c),$ so
we may rewrite $H$ as
\begin{equation}\nonumber
H(p)=-p(u')^{-1}(p)+u((u')^{-1}(p)).
\end{equation}
We may then also calculate $H_{p},$ which is given by the formula
\begin{equation}\nonumber
H_{p}(p)=-(u')^{-1}(p)-\frac{p}{u''((u')^{-1}(p))}+\frac{p}{u''((u')^{-1}(p))}=-(u')^{-1}(p).
\end{equation}
Since we have taken $u$ to be smooth, we see that $H$ and $H_{p}$ inherit this smoothness.

The above calculation requires $p>0;$ if instead $p\leq0,$ then there is no maximum, and the Hamiltonian
would have the value $+\infty.$  To restrict to $p>0$ we must take $\partial_{a}v>0,$ and thus it is convenient
to change variables to $w=\partial_{a}v$ and seek positive solutions for $w.$  We furthermore wish to have compactly
supported solutions, and this is not possible with the condition we have just stated, that $w>0$ on the whole domain.
So, we introduce $y=w-f(t)w_{\infty}$ for some positive constant $w_{\infty},$ and we require $y$ to be smooth and
compactly supported.  We will likewise require $g$ to be compactly supported.

We let $y=\partial_{a}v-f(t)w_{\infty},$ and seek a favorable choice of the function $f(t).$
We need to determine the equation satisfied by $y$ and also to choose our $f.$  To this end, we begin by
differentiating \eqref{vEquation} with respect to $a$
\begin{multline}\label{vaEquation}
\partial_{t}(\partial_{a}v)+\frac{1}{2}\sigma^{2}(z)\partial_{zz}(\partial_{a}v)+\mu(z)\partial_{z}(\partial_{a}v)
\\
+r(t)\partial_{a}v+(z+r(t)a)\partial_{a}(\partial_{a}v)+H_{p}(\partial_{a}v)\partial_{a}(\partial_{a}v)-\rho\partial_{a}v
=0.
\end{multline}
To each $\partial_{a}v$ appearing on the right-hand side, we add and subtract $f(t)w_{\infty}.$
We find the following evolution equation for $y:$
\begin{multline}\nonumber
\partial_{t}y + f'(t)w_{\infty} + \frac{1}{2}\sigma^{2}(z)\partial_{zz}y + \mu(z)\partial_{z}y + r(t)y 
+r(t)f(t)w_{\infty} 
\\
+ (z+r(t)a)\partial_{a}y + \Theta(y,f)\partial_{a}y -\rho y - \rho f(t)w_{\infty}=0.
\end{multline}
Here we have introduced $\Theta$ to be the function given by
\begin{equation}\nonumber
\Theta(y,f)=H_{p}(y+fw_{\infty}).
\end{equation}
We choose $f$ such that
\begin{equation}\label{fFinalEquation}
f'(t)+r(t)f(t)-\rho f(t) =0;
\end{equation}
note that this is a simple ordinary differential equation which may be solved with an integrating factor.
We also must specify a terminal condition for $f,$ and we take $f(T)=1.$
This choice leaves the equation for $y$ as
\begin{equation}\label{yFinalEquation}
\partial_{t}y+\frac{1}{2}\sigma^{2}(z)\partial_{zz}y+\mu(z)\partial_{z}y
+(r(t)-\rho)y+(z+r(t)a+\Theta(y,f))\partial_{a}y = 0.
\end{equation}
In terms of $y$ and $f,$ and thus also in terms of $\Theta,$ our equation for $g$ is
\begin{equation}\label{gFinalEquation}
\partial_{t}g-\frac{1}{2}\partial_{zz}(\sigma^{2}(z)g)+\partial_{z}(\mu(z)g)+\partial_{a}((z+r(t)a+\Theta(y,f))g)
=0.
\end{equation}

\section{Determining the interest rate, and a relaxed problem}\label{interestRateSection}

In this section we explore the nature of the coupling between the $v$ equation \eqref{vEquation} and the
$g$ equation \eqref{gEquation}.  We will proceed first in terms of $v,$ and then summarize in terms of our 
new variable $y.$  As stated in \cite{mollPhilTrans}, the coupling is through the interest rate, $r(t),$ and this
interest rate is determined through the moment condition \eqref{gFirstMoment}.  

We proceed with our first calculation on this point, which we expect is what was intended in \cite{mollPhilTrans}.
We assume that \eqref{gFirstMoment} is satisfied by the data $g_{0}.$

Call $\mathcal{C}=\int\int ag\ dadz.$  Then we differentiate $\mathcal{C}$ with respect to time:
\begin{multline}\nonumber
\mathcal{C}_{t}=\int\int\frac{a}{2}\partial_{zz}(\sigma^{2}g)\ dadz 
-\int\int a\partial_{z}(\mu g)\ dadz
\\
-\int\int a\partial_{a}((z+ra)g)\ dadz 
-\int\int a\partial_{a}(H_{p} g)\ dadz.
\end{multline}
By assumptions on the diffusion and drift coefficients $\sigma$ and $\mu,$ the first and second terms on the right-hand 
side vanish.  For the third and fourth terms on the right-hand side, we integrate by parts:
\begin{multline}\nonumber
\mathcal{C}_{t}=-\int a(z+ra)g\Bigg|_{a_{min}}^{a=\infty}\ dz + \int\int(z+ra)g\ dadz
\\
-\int a H_{p}g\Bigg|_{a_{min}}^{a=\infty}\ dz+\int\int H_{p}g\ dadz.
\end{multline}
Because of our assumption of compact support with respect to $a$ in $(a_{min},\infty),$ the first and third terms
on the right-hand side also vanish.  This leaves us with 
\begin{equation}\label{unfortunateR}
\mathcal{C}_{t}-r(t)\mathcal{C}=\mathcal{Q},
\end{equation}
with the quantity $\mathcal{Q}$ defined by $\mathcal{Q}=\int\int (z+H_{p})g\ dadz.$

Unfortunately this is a difficulty, as it is unclear from this how to determine $r$ from \eqref{unfortunateR}.
That is, if we believe that $r$ will enforce $\mathcal{C}=0,$ then we must have $\mathcal{C}_{t}=0$ as well,
and then \eqref{unfortunateR} tells us that $\mathcal{Q}$ must equal zero as well.  However this would not
tell us what the interest rate is actually equal to.  Worse yet, there is no reason to believe at present that $\mathcal{Q}$
would equal zero.  We deal with this difficulty by generalizing the problem.  Instead of seeking solutions for which
$\mathcal{C}=0,$ we now will determine the interest rate by insisting $\mathcal{Q}_{t}=0.$

\begin{remark}
Note that if $g|_{a_{min}}\neq0,$ then there would be another term proportional to $r$ in \eqref{unfortunateR}.
It would then be possible to choose a value of $r$ to cancel the $\mathcal{Q}$ term.  
\end{remark}

As we have just said, the condition $\mathcal{C}=0$ does indeed imply 
$\mathcal{Q}=0$ and thus $\mathcal{Q}_{t}=0.$  Thus solutions of the original problem ($\mathcal{C}=0$)
also solve the relaxed problem ($\mathcal{Q}_{t}=0$).  In the other direction, if we have a solution
of the relaxed problem, since $\mathcal{Q}_{t}=0$ we have $\mathcal{Q}=\mathcal{Q}_{0}$ for all $t.$  If 
$\mathcal{Q}_{0}=0$ and if $\mathcal{C}(0)=0,$ then we may conclude that $\mathcal{C}=0$ after all.
If however $\mathcal{Q}_{0}\neq0$ and if $\mathcal{C}(0)=0,$ then we see that $\mathcal{C}_{t}(0)\neq0$ and
thus $\mathcal{C}$ is not identically zero.  

We will be proving existence and uniqueness of solutions for the relaxed problem.  Thus if there is a solution of the
original problem, then it must be the solution we prove to exist.  We will in some cases be able to guarantee that
in fact $\mathcal{Q}_{0}\neq0,$ and thus in these cases, the original problem does not have a solution.

Now that we are considering the relaxed problem, we return our attention to determination of the interest rate.
Taking the time derivative of $\mathcal{Q},$ we have
\begin{multline}\nonumber
\mathcal{Q}_{t}=\int\int z\partial_{t}g\ dadz + 
\int\int H_{pp}(\partial_{a}v)(\partial_{t}\partial_{a}v)g\ dadz
\\
+\int\int H_{p}(\partial_{a}v)(\partial_{t}g)\ dadz :=Q_{1}+Q_{2}+Q_{3}.
\end{multline}
For each of these terms, we decompose into a part which explicitly involves $r$ and a piece which does not:
\begin{eqnarray}
\label{Q1Equation}
Q_{1} & = & P_{1} - \int\int z\partial_{a}\left((z+r(t)a)g\right)\ dadz,\\
\label{Q2Equation}
Q_{2} & = & P_{2} - \int\int g(H_{pp}(\partial_{a}v))\partial_{a}\left((z+r(t)a)\partial_{a}v\right)\ dadz,\\
\label{Q3Equation}
Q_{3} & = & P_{3} - \int\int (H_{p}(\partial_{a}v))\partial_{a}\left((z+r(t)a)g\right)\ dadz,
\end{eqnarray} 
where
\begin{equation}\nonumber
P_{1}=-\int\int z\partial_{z}(\mu(z)g)\ dadz,
\end{equation}
\begin{equation}\nonumber
P_{2}=\int\int gH_{pp}(\partial_{a}v)\left(-\frac{\sigma^{2}(z)}{2}\partial_{zz}(\partial_{a}v)
-\mu(z)\partial_{z}(\partial_{a}v)-H_{p}(\partial_{a}v)\partial_{a}(\partial_{a}v)+\rho\partial_{a}v\right)\ dadz,
\end{equation}
\begin{equation}\nonumber
P_{3}=\int\int H_{p}(\partial_{a}v)\left(\frac{1}{2}\partial_{zz}(\sigma^{2}(z)g)-\partial_{z}(\mu(z)g)
-\partial_{a}(gH_{p}(\partial_{a}v))\right)\ dadz.
\end{equation}
We first notice that, because of the compact support with respect to $a$ in $(a_{min},\infty),$ the integral
on the right-hand side of \eqref{Q1Equation} is equal to zero.  We apply the derivative in the integral on the
right-hand side of \eqref{Q2Equation}, and we integrate by parts in \eqref{Q3Equation}:
\begin{multline}\label{Q2last}
Q_{2}  =  P_{2} - r(t)\int\int g(H_{pp}(\partial_{a}v))\partial_{a}v\ dadz\\
- \int\int g(H_{pp}(\partial_{a}v))(z+r(t)a)\partial_{a}^{2}v\ dadz,
\end{multline}
\begin{equation}\label{Q3last}
Q_{3}  =  P_{3} + \int\int (H_{pp}(\partial_{a}v))(\partial_{a}^{2}v)(z+r(t)a)g\ dadz.
\end{equation}
We introduce the notation $P=P_{1}+P_{2}+P_{3},$ and 
\begin{equation}\nonumber
K=\int\int g(H_{pp}(\partial_{a}v))\partial_{a}v\ dadz.
\end{equation}
Then adding $Q_{1},$ $Q_{2},$ and $Q_{3}$ back together again, we find
\begin{equation}\nonumber
\mathcal{Q}_{t}=P-r(t)K;
\end{equation}
to arrive at this, notice that there is a cancellation when adding \eqref{Q2last} and \eqref{Q3last}.
We therefore have concluded that we may determine $r(t)$ in the relaxed problem by
\begin{equation}\nonumber
r(t)=\frac{P}{K}.
\end{equation}
(Note that both $P$ and $K$ depend on time.)

For this to be a complete description of the determination of the interest rate, we must do two further things.
First, we remark that it is clear that $K$ is nonzero.  Since $H_{p}(p)=-(u')^{-1}(p)$ and since $u'$ is strictly decreasing, 
we see that $H_{pp}(p)>0$ always.  As discussed above, we are only considering solutions for which $\partial_{a}v>0.$
Together with the fact that $g$ is a probability distribution, we have $K>0.$  We will still, however, need to 
control $K$ to ensure that it cannot get arbitrarily small.  Finally, we give an explicit formula for $P,$
in terms of $y$ and $f$ rather than $\partial_{a}v:$ 
\begin{multline}\label{definitionOfP}
P=P[y,f,g]=-\int\int z\partial_{z}(\mu(z)g)\ dadz
\\
+\int\int gH_{pp}(y+fw_{\infty})\left(-\frac{1}{2}\sigma^{2}\partial_{zz}y-\mu\partial_{z}y-H_{p}(y+fw_{\infty})\partial_{a}y
+\rho (y+fw_{\infty})\right)\ dadz\\
+\int\int H_{p}(y+fw_{\infty})\left(\frac{1}{2}\partial_{zz}(\sigma^{2}g)-\partial_{z}(\mu g)-\partial_{a}(gH_{p}(y+fw_{\infty}))
\right)\ dadz.
\end{multline}

\section{Iterative scheme}\label{iterativeSection}

We will prove our existence theorem using an iterative scheme, and we will now set up this scheme. 
 
We fix $s\in\mathbb{N}$ such that $s\geq 4;$ we will provide some further comments on this later.
Let $A>0$ be given.  We let 
$A_{1}=[-A,A],$ $A_{2}=[-2A,2A],$ and $A_{3}=[-3A,3A].$  We let $\chi$ be such that $\chi\in C^{\infty}(\mathbb{R}),$
such that $\chi(a)=1$ for $a\in A_{2},$ such that $\chi(a)=0$ for $a\in A_{3}^{c},$ and such that on each component of
$A_{3}\setminus A_{2},$
$\chi$ is smooth and monotone.  For all $a\in \mathbb{R},$ we then have $|\chi(a)a|\leq 3A.$
We will henceforth work in the spatial domain which we denote by $D,$ which is $D=A_{3}\times[z_{min},z_{max}].$

Let data $g_{0}\in H^{s}(D)$ and 
$y_{T}\in H^{s+1}(D)$ be given, such that the support of $g_{0}$ with respect to $a$ is contained in the interior of
$A_{1}$ and the support of $y_{T}$ with respect to $a$ is contained in the interior of $A_{1}.$
We initialize our scheme with $g^{0}=g_{0,\delta}$ and $y^{0}=y_{T,\delta}.$  
Here, for small parameter values 
$\delta>0,$ we have taken a $C^{\infty}$ function $g_{0,\delta}$ and a $C^{\infty}$ function $y_{T,\delta}$
to be within $\delta$ of $g_{0}$ in $H^{s}(D)$ and within $\delta$ of $y_{T}$ in $H^{s+1}(D),$ respectively. 
As we have assumed that $g_{0}$ and $y_{T}$ are 
each supported in the interior of $A_{1}$ with respect to the $a$ variable, we may take
our approximations to also be supported in this set with respect to $a.$  That our data can be approximated in this
way follows from standard density results \cite{adams}.

The solutions of our iterated system will
actually depend on both $n$ and $\delta$ and would more properly be called $y^{n,\delta}$ and
$g^{n,\delta};$ we will suppress this $\delta$
dependence, however, for the time being, considering for now $\delta>0$ to be fixed, and we will call the iterates 
$y^{n}$ and $g^{n},$ and so on.
We take the function $f^{0}(t)=1$ for all $t,$ and we let
the initial interest rate be given as $r^{0}(t)=0$ for all $t.$   
We will still need to initialize $K.$

For our constant $w_{\infty}>0$ and the data $y_{T}$ we define
\begin{equation}\label{dataW}
W=\min_{(a,z)\in D}\left(y_{T}(a,z)+w_{\infty}\right),
\end{equation}
and we require that $W>0;$ this is the positivity condition for $\partial_{a}v.$
Noting that our terminal data in our approximate problems is 
not exactly equal to $y_{T}+w_{\infty},$ we also take $\delta$ sufficiently small so that
\begin{equation}\label{deltaData1}
\min_{(a,z)\in D}\left(y_{T,\delta}+w_{\infty}\right)\geq \frac{3W}{4}.
\end{equation}
We similarly define $K_{data}>0$ as
\begin{equation}\nonumber
K_{data}=\int\int g_{0}(H_{pp}(y_{T}+w_{\infty}))(y_{T}+w_{\infty})\ dadz.
\end{equation}
Note that $K_{data}$ is positive since $g_{0}$ is a probability distribution, since $H_{pp}>0$ (this sign is inherited
from properties of the utility function, $u$) and because we have taken $W>0.$
We need to initialize $K$ and use something like $K_{data},$ but adapted to the data for our approximate problems,
\begin{equation}\nonumber
K^{0}=\int\int g_{0,\delta}(H_{pp}(y_{T,\delta}+w_{\infty})(y_{T,\delta}+w_{\infty})\ dadz,
\end{equation}
and we may take $\delta$ sufficiently small so that
\begin{equation}\label{deltaData2}
K^{0}\geq \frac{3K_{data}}{4}.
\end{equation}

Having initialized our iteration scheme with initial iterates $y^{0}=y_{T,\delta}$ and 
$g^{0}=g_{0,\delta},$
the support of each of $y^{0}$ and $g^{0}$ with respect to $a$ is contained in $A_{1}$ and thus also in $A_{2}.$
We fix $M>1.$  We may take $\delta>0$ sufficiently small so that 
we also have the following bounds for $y^{0}$ and $g^{0}:$
\begin{equation}\nonumber
\sup_{t\in[0,T]}\|y^{0}(t,\cdot)\|^{2}_{H^{s+1}}
+\|g^{0}(t,\cdot)\|^{2}_{H^{s}}
\leq M\Big(\|y_{T}\|^{2}_{H^{s+1}}+\|g_{0}\|^{2}_{H^{s}}\Big).
\end{equation}
These two bounds, on the supports and on the norms, are features we will seek to maintain for all subsequent iterates.

We introduce another cutoff function, related to the fact that 
the function $H_{p}$ is only defined for positive arguments.  We have given the definition of $W>0$ above in
\eqref{dataW}.  We let $\psi:\mathbb{R}\rightarrow\mathbb{R}$ be a $C^{\infty}$ function which satisfies
$\psi(x)=x$ for $x\geq W/2,$ which satisfies $\psi(x)=W/4$ for $x\leq W/4,$ and which is monotone.
We define $\Theta_{c}$ by
\begin{equation}\nonumber
\Theta_{c}(y,f)=H_{p}(\psi(y+fw_{\infty})).
\end{equation}
It will be important later to note that if $y+fw_{\infty}\geq W/2,$ then $\Theta_{c}(y,f)=\Theta(y,f).$

We set up our iterative scheme, beginning with $g:$
\begin{multline}\label{gIteratedTransport}
\partial_{t}g^{n+1}-\frac{1}{2}\partial_{zz}\left(\sigma^{2}(z)g^{n+1}\right)
\\
+\partial_{z}\left(\mu(z)g^{n+1}\right)
+\partial_{a}\left(\chi(z+r^{n}(t)a)g^{n+1}\right)+\partial_{a}\left(g^{n+1}\chi\Theta_{c}(y^{n},f^{n})\right)=0.
\end{multline}
We take this with initial data
\begin{equation}\label{mollData}
g^{n+1}(0,\cdot)=g_{0,\delta}.
\end{equation}
Note that we have inserted a factor of the cutoff function $\chi$ in the transport terms.  A difficulty of the system
is that as long as $r\neq0,$ the transport speeds are unbounded.  With the factors of $\chi$ present, this is no
longer the case for our approximate equations.  We will be able to remove the factors of $\chi$ by the end of our
existence argument.

The transport speed in \eqref{gIteratedTransport} with respect to the variable $a,$ then, is 
$\chi z+r^{n}(t)\chi a+\chi\Theta_{c}(y^{n},f^{n}).$ 
Denote by $R$ an upper bound on $r^{n}(t),$ and 
denote by $Y$ and upper bound on $\Theta(y^{n},f^{n}),$ presuming for the moment that these bounds can
be found independent of our parameters $n$ and $\delta.$
Then the transport speed is bounded by $z_{max}+3RA+Y,$ independently of $n$ and $\delta.$
Thus, until time $T,$ the support of $g^{n+1}$ with respect to $a,$ which is initially contained in
$A_{1},$ remains contained in $A_{2}$ as long as $T\leq\frac{A}{z_{max}+3RA+Y}.$

We next give the iterated equation for $y:$
\begin{multline}\label{yIteratedTransport}
\partial_{t}y^{n+1} + \frac{1}{2}\sigma^{2}(z)\partial_{zz}y^{n+1} + \mu(z)\partial_{z}y^{n+1} + r^{n}(t)y^{n+1} 
\\
+ (\chi z+r^{n}(t)\chi a)\partial_{a}y^{n+1} + \chi\Theta_{c}(y^{n},f^{n})\partial_{a}y^{n+1} -\rho y^{n+1} =0.
\end{multline}
As above, we take this with mollified data 
\begin{equation}\label{yMollData}
y^{n+1}(T,\cdot)=y_{T,\delta}.
\end{equation}
Again, the solutions may more properly be called $y^{n,\delta},$ but we will suppress the $\delta$ dependence for the time 
being.  Note that we have the same transport speed with respect to $a$ as in the $g^{n+1}$ equation, and therefore
we have the same support properties; with initial data supported in $A_{1},$ and with the presumed upper bounds, 
the support of $y^{n+1}$ remains in $A_{2}$ as long as $T\leq\frac{A}{z_{max}+3RA+Y}.$

To finish specifying the iterated problem, we must specify $f^{n+1}$ and $r^{n+1},$ and the latter of these will require
specifying $P^{n+1}$ and $K^{n+1}.$    We take $f^{n+1}$ to be the solution of the ordinary differential equation
\begin{equation}\label{fIteratedEquation}
(f^{n+1})'(t)+r^{n}(t)f^{n+1}(t)-\rho f^{n+1}(t)=0,
\end{equation}
with terminal condition $f(T)=1.$
Notice that the solution of this terminal value problem is
\begin{equation}\label{fIteratedExplicit}
f^{n+1}(t)=\exp\left\{\int_{t}^{T}r^{n}(t')-\rho\ dt'\right\}.
\end{equation}

We take $r^{n+1}$ to be given by
\begin{equation}\label{rIterated}
r^{n+1}=\frac{P^{n}}{K^{n}},
\end{equation}
where we need to define $P^{n}$ and $K^{n}.$
Consistent with our previous definition of $K$ we denote
\begin{equation}\nonumber
K[y,g,f]=\int\int g(H_{pp}(y+fw_{\infty}))(y+fw_{\infty})\ dadz;
\end{equation}
but $K[y^{n},g^{n},f^{n}]$ 
is not sufficient for use in our iterative scheme because we need to use the cutoff function $\psi.$
Thus, for any value of $n,$ given $y^{n},$ $f^{n},$ and $g^{n},$ we define $K^{n+1}$ as
\begin{equation}\label{KIteratedDefinition}
K^{n+1}=\int\int g^{n}(H_{pp}(\psi(y^{n}+f^{n}w_{\infty})))(y^{n}+f^{n}w_{\infty})\ dadz.
\end{equation}
Finally, recalling $P[y,f,g]$ as defined in \eqref{definitionOfP}, we must introduce a version $P_{c}$ which involves
the cutoff function $\psi:$
\begin{multline}\nonumber
P_{c}[y,f,g]=
-\int\int z\partial_{z}(\mu(z)g)\ dadz
\\
+\int\int gH_{pp}(\psi(y+fw_{\infty}))
\left(-\frac{1}{2}\sigma^{2}\partial_{zz}y-\mu\partial_{z}y-\Theta_{c}(y,f)\partial_{a}y
+\rho(y+fw_{\infty})\right)\ dadz\\
+\int\int \Theta_{c}(y,f)
\left(\frac{1}{2}\partial_{zz}(\sigma^{2}g)-\partial_{z}(\mu g)-\partial_{a}(g\Theta_{c}(y,f))
\right)\ dadz.
\end{multline}
We can then define our iterated $P$ as
\begin{equation}\nonumber
P^{n}=P_{c}[y^{n},f^{n},g^{n}].
\end{equation}

\section{Existence and bounds for the iterates}\label{existenceAndBounds}

In order to eliminate our approximation parameters, i.e. send $n\rightarrow\infty$ and $\delta\rightarrow0,$
we need to establish bounds for the iterates which are uniform with respect to $n$ and $\delta.$
We fix a value $M>1$ and we
assume the following are satisfied by the $n$-th iterates:
\begin{equation}\label{yInductiveHypothesis}
\|y^{n}\|_{H^{s+1}}\leq M\|y_{T}\|_{H^{s+1}},
\end{equation}
\begin{equation}\label{gInductiveHypothesis}
\|g^{n}\|_{H^{s}}\leq M\|g_{0}\|_{H^{s}},
\end{equation}
\begin{equation}\label{fInductiveHypothesis}
f^{n}\in\left[\frac{1}{2},2\right],\quad\forall t\in[0,T],
\end{equation}
\begin{equation}\label{KInductiveHypothesis}
K^{n}\geq \frac{K_{data}}{2},\quad\forall t\in[0,T].
\end{equation}
We furthermore assume that the $n$-th iterates are infinitely smooth.

Based on these values, we define a value $P_{max};$ we take this to be the supremum of
the set of values $\{|P_{c}[\tilde{y},\tilde{f},\tilde{g}]|\},$ where $\tilde{y},$ $\tilde{g},$ and $\tilde{f}$ satisfy
\begin{equation}\nonumber
\|\tilde{y}\|_{H^{s+1}}\leq M\|y_{T}\|_{H^{s+1}},\qquad
\|\tilde{g}\|_{H^{s}}\leq M\|g_{0}\|_{H^{s}},\qquad
\tilde{f}\in[1/2,2].
\end{equation}
With this definition, we then have our inductive hypothesis for the iterates for the interest rate:
\begin{equation}\label{rInductiveHypothesis}
r^{n}\in\left[-\frac{2P_{max}}{K_{data}},\frac{2P_{max}}{K_{data}}\right],\quad\forall t\in[0,T].
\end{equation}

Finally, we have one more condition we wish to have satisfied for our iterates, and that is the positivity condition for
$\partial_{a}v.$  Recall the definition of $W>0$ in \eqref{dataW}.
Then we desire that the following condition is satisfied for $y^{n}$ and $f^{n}:$
\begin{equation}\label{WInductiveHypothesis}
\min_{(t,a,z)\in[0,T]\times D} \left(y^{n}(t,a,z)+f^{n}(t)w_{\infty}\right) \geq \frac{W}{2}.
\end{equation}

Note that with our specification of the initial iterates, the bounds 
\eqref{yInductiveHypothesis}, \eqref{gInductiveHypothesis}, \eqref{fInductiveHypothesis}, and \eqref{rInductiveHypothesis}
are satisfied for $n=0.$  
By \eqref{deltaData1} we have satisfied \eqref{WInductiveHypothesis} as well for $n=0.$
Similarly, by \eqref{deltaData2}, we have satisfied \eqref{KInductiveHypothesis} when $n=0.$
We may also note that all of the initial iterates are in $C^{\infty}.$
We must verify that each of \eqref{yInductiveHypothesis}, \eqref{gInductiveHypothesis},
\eqref{fInductiveHypothesis}, \eqref{KInductiveHypothesis}, \eqref{rInductiveHypothesis}, and 
\eqref{WInductiveHypothesis} are satisfied for the $(n+1)$-st iterates, but first we must ensure
that the $(n+1)$-st iterates exist.

\begin{lemma}\label{gIteratedLemma} Let $T>0,$ and 
let $y^{n},$ $g^{n},$ $r^{n},$ $f^{n},$ and $K^{n}$ be as described above, on the time interval $[0,T].$
There exists a unique $C^{\infty}$ solution $g^{n+1}$ to the initial value problem
\eqref{gIteratedTransport}, \eqref{mollData} on the time interval $[0,T].$
\end{lemma}
\begin{proof}
We prove existence by the energy method, the steps of which are to introduce mollifiers, 
use the Picard theorem to get existence of solutions, prove an estimate uniform with respect to the mollification
parameter, and then pass to the limit as the mollification parameter vanishes.  
To use standard theory of mollifiers, we first replace our spatial domain with a torus.

We make an extension of the domain in the $z$ variable.  Let $\omega\in\mathbb{N}$ be any
finite degree of regularity, sufficiently large.
We take $\widetilde{\sigma},$ $\widetilde{\mu},$ and $\widetilde{\Theta}$ to be $H^{\omega+2}$ extensions
of $\sigma,$ $\mu,$ and $\Theta(y^{n},f^{n})$ to the domain $[z_{min}-3, z_{max}+3]$ (for $\sigma$ and $\mu$)
and to the domain $A_{3}\times[z_{min}-3, z_{max}+3]$ (for $\Theta$).  There are many versions of the existence of such
extensions available in the literature, and we cite \cite{millerExtension} in particular.
We let $\phi$ be a cutoff function which is equal to $1$ for $z\in[z_{min}-1, z_{max}+1]$
and which is equal to zero on $[z_{min}-3,z_{min}-2]$ and on $[z_{max}+2,z_{max}+3],$ and which is smooth
and monotone on the remaining components of the new $z$ domain.  In writing an evolution equation
to approximate \eqref{gIteratedTransport}, we will replace
$\sigma,$ $\mu,$ and $\Theta(y^{n},f^{n})$ with $\phi\widetilde{\sigma},$ $\phi\widetilde{\mu},$
and $\phi\widetilde{\Theta},$ respectively.  We also replace the transport coefficient
$\chi(z+r^{n}a)$ with $\phi\chi(z+r^{n}a).$  We take $\widetilde{g}_{0}$ to be an $H^{\omega}$ extension of
$g_{0,\delta},$ and we will use data $\phi\chi\widetilde{g}_{0}.$

The coefficients in our new evolution equation, because they are zeroed out at the ends of the interval
$[z_{min}-3, z_{max}+3],$ are periodic with respect to $z.$  Similarly, the coefficients are all also periodic with respect
to $a$ on $A_{3}.$  Because of the presence of $\chi$ and $\phi$ in our proposed data, we also have periodic 
initial data.  We call our new domain $\widetilde{D},$ and we consider this now to be a torus, i.e. we take
periodic boundary conditions.  We let $\mathcal{J}_{\tau}$ be a standard mollifier on the two-dimensional torus
with parameter $\tau>0.$  We introduce an approximate equation:
\begin{multline}\label{hEvolution}
\partial_{t}h^{\tau}-\frac{1}{2}\partial_{zz}\mathcal{J}_{\tau}((\phi\widetilde{\sigma})^{2}\mathcal{J}_{\tau}h^{\tau})
\\
+\partial_{z}\mathcal{J}_{\tau}((\phi\widetilde{\mu})\mathcal{J}_{\tau}h^{\tau})+
\partial_{a}\mathcal{J}_{\tau}(\phi\chi(z+r^{n}a)\mathcal{J}_{\tau}h^{\tau})
+\partial_{a}\mathcal{J}_{\tau}(\phi\chi\widetilde{\Theta}\mathcal{J}_{\tau}h^{\tau})=0.
\end{multline}
As we have said, we take this evolution with initial condition
\begin{equation}\nonumber
h(0,\cdot)=\phi\chi\widetilde{g}_{0}.
\end{equation}
The presence of the mollifiers turns all derivatives on the right-hand side of \eqref{hEvolution} into bounded operators;
the Picard Theorem \cite{majdaBertozzi} 
then implies that there exists a solution for a time $T_{\tau}>0.$  This solution may
be continued as long as the solution does not blow up; in this case, an energy estimate, using standard mollifier 
properties and integration by parts, implies that the $H^{\omega}(\widetilde{D})$ norm of $h$ does not blow up on $[0,T].$
We introduce an energy, equivalent to the square of the $H^{\omega}(\widetilde{D})$ norm,
\begin{equation}\nonumber
E(t)=\sum_{j=0}^{\omega}\sum_{\ell=0}^{\omega-j}E_{j,\ell}(t),
\qquad
E_{j,\ell}(t)=\frac{1}{2}\int_{\widetilde{D}}\left(\partial_{a}^{j}\partial_{z}^{\ell}h^{\tau}(t,a,z)\right)^{2}\ dadz.
\end{equation}
Taking the time derivative of the energy, using the facts that $\mathcal{J}_{\tau}$
commutes with derivatives and is self-adjoint, and using other
mollifier properties such as $\|\mathcal{J}_{\tau}f\|_{H^{m}}\leq\|f\|_{H^{m}}$ for any $f$ and any $m,$ and integrating
by parts yields the conclusion
\begin{equation}\label{tauEnergyEstimate}
\frac{dE}{dt}\leq cE,
\end{equation}
where $c$ is independent of $\tau.$ (We do not provide further 
details of this energy estimate as it is very similar to the 
estimate in Theorem \ref{uniformBoundTheorem} below).  
The bound \eqref{tauEnergyEstimate} implies that the solutions $h^{\tau}$
are uniformly bounded in $H^{\omega}(\widetilde{D})$ with respect to the approximation parameter $\tau,$ 
and that our solutions $h^{\tau}$ all exist on the common time interval $[0,T].$

The uniform bound implies that the first derivatives of the solutions with respect to $a,$ $z,$ and $t$ are all uniformly
bounded, and thus our solutions $h^{\tau}$ form an equicontinuous family.
Thus there is a uniformly convergent subsequence (which we do not relabel), 
as $\tau$ vanishes; we call the limit $h.$
Uniform convergence implies convergence in $L^{2}$ in a bounded domain, so we see that
$h^{\tau}$ converges to $h$ in $C([0,T];L^{2}(\widetilde{D})).$  Using the uniform bound 
in $H^{\omega}(\widetilde{D}),$
a standard Sobolev interpolation theorem (see \cite{ambroseThesis}, for example) then implies convergence
in $C([0,T];H^{\omega-1}(\widetilde{D})).$  Furthermore the uniform bound implies that we have a weak limit
at every time in $H^{\omega},$ and this weak limit must be $h,$ so we have $h\in L^{\infty}([0,T];H^{\omega})$
as well.

Taking the integral with respect to time of \eqref{hEvolution} and then passing to the limit as $\tau$ 
vanishes (this is possible because of the regularity we have established, including convergence in 
$C([0,T];H^{\omega-1})$), and then differentiating with respect to time, we see that $h$ satisfies
\begin{equation}\label{hEquationAfterLimit}
\partial_{t}h-\frac{1}{2}\partial_{zz}((\phi\widetilde{\sigma})^{2}h)
+\partial_{z}(\phi\widetilde{\mu}h)+\partial_{a}(\phi\chi(z+r^{n}a)h)+\partial_{a}(\phi\chi\widetilde{\Theta}h)=0.
\end{equation}
When taking this limit we again use various standard mollifier properties; a good list of such properties can be found
in Lemma 3.5 of \cite{majdaBertozzi}.  Perhaps the most useful of these to arrive at \eqref{hEquationAfterLimit}
is, for any $m\in\mathbb{N},$ 
\begin{equation}\nonumber
\|\mathcal{J}_{\tau}f-f\|_{H^{m}}\leq\tau\|f\|_{H^{m+1}}.
\end{equation}
We define $g^{n+1}$ to be the restriction of $h$ to the domain $D.$  On $D,$ we have $\phi=1,$
$\widetilde\sigma=\sigma,$ $\widetilde{\mu}=\mu,$ $\widetilde{\Theta}=\Theta(y^{n},f^{n}),$ and
$\widetilde{g}_{0}=g_{0,\delta}.$  Furthermore on $D$ we also have 
$\chi g_{0,\delta}=g_{0,\delta}.$  We conclude that $g^{n+1}$ satisfies
\eqref{gIteratedTransport} and \eqref{mollData}.

We have two further points to make, to complete the proof.  First, we mention that uniqueness of solutions of 
the initial value problem \eqref{gIteratedTransport}, \eqref{mollData} is straightforward. The initial value problem
satisfied by the difference of two solutions is a linear equation with zero forcing and zero data, and an
estimate in $L^{2}$ for the difference of two smooth solutions can be made.
Finally, on regularity, we mention that the regularity parameter $\omega$ was arbitrary, so we see that the solution
$g^{n+1}$ is infinitely smooth with respect to the spatial variables.  Upon taking higher derivatives of 
\eqref{gIteratedTransport} with respect to time, it can be seen that the solutions are also infinitely smooth with 
respect to time.  This completes the proof.
\end{proof}

We also have existence of the iterated $y^{n+1},$ given in the following lemma.
\begin{lemma}\label{yIteratedLemma}
Let $y^{n},$ $g^{n},$ $r^{n},$ $f^{n},$ and $K^{n}$ be as described above.
There exists a unique $C^{\infty}$ solution $y^{n+1}$ to the initial value problem
\eqref{yIteratedTransport}, \eqref{yMollData} on the time interval $[0,T].$
\end{lemma}
We omit the proof of Lemma \ref{yIteratedLemma}, 
as the method is entirely the same as that of Lemma \ref{gIteratedLemma}.

To conclude this section, 
we mention that it is immediate from their definitions and the smoothness assumptions on the $n$-th iterates
that $f^{n+1},$ $K^{n+1},$ and $r^{n+1}$ are $C^{\infty}$ in time.

\subsection{Uniform Bounds}\label{uniformSection}
Recall that we have fixed $s\in\mathbb{N}$ satisfying $s\geq4,$ and we have taken  
$g_{0}\in H^{s}$ and $y_{T}\in H^{s+1}.$  The requirement $s\geq 4$ will guarantee that the solutions we find
are classical solutions of the PDE system, and will allow us to use Sobolev embedding and related inequalities
as needed.  Note that while we have demonstrated above that the iterates are infinitely smooth, this has relied on 
the $C^{\infty}$ approximation $g_{0,\delta}$ to the intended data $g_{0};$ with the data $g_{0}\in H^{s}$
and $y_{T}\in H^{s+1},$ we can only expect bounds on the iterates which are uniform with respect to the 
parameters in these spaces.

\begin{theorem}\label{uniformBoundTheorem}
There exists $T_{*}>0$ such that if the time horizon satisfies
$T\in(0,T_{*}),$ then for all $n\in\mathbb{N}$ and for all $\delta>0,$
the iterates $(y^{n},g^{n},f^{n},K^{n},r^{n})$ defined above satisfy 
\eqref{yInductiveHypothesis}, \eqref{gInductiveHypothesis}, \eqref{fInductiveHypothesis}, \eqref{KInductiveHypothesis},
\eqref{rInductiveHypothesis}, and \eqref{WInductiveHypothesis}.
\end{theorem}

\begin{proof}
The proof will be by induction.  We have remarked previously that \eqref{yInductiveHypothesis}, \eqref{gInductiveHypothesis}, \eqref{fInductiveHypothesis}, \eqref{KInductiveHypothesis},
\eqref{rInductiveHypothesis}, and \eqref{WInductiveHypothesis} hold in the case $n=0;$ this is the base case.
The statements \eqref{yInductiveHypothesis}, \eqref{gInductiveHypothesis}, \eqref{fInductiveHypothesis}, \eqref{KInductiveHypothesis},
\eqref{rInductiveHypothesis}, and \eqref{WInductiveHypothesis} then together constitute the inductive hypothesis.

We begin by determining a bound for the next iterate $g^{n+1}.$
We let the functional $E_{j,\ell}$ be given by
\begin{equation}\nonumber
E_{j,\ell}(t)=\frac{1}{2}\int_{D}\left(\partial_{a}^{j}\partial_{z}^{\ell}g^{n+1}\right)^{2}\ dadz,
\end{equation}
and we sum over $j$ and $\ell$ to form the energy $E(t),$
\begin{equation}\nonumber
E(t)=\sum_{j=0}^{s}\sum_{\ell=0}^{s-j}E_{j,\ell}(t).
\end{equation}
Of course, the energy $E$ is equivalent to the square of the $H^{s}$-norm of $g^{n+1}.$

We will now demonstrate a bound for the growth of the energy.
For given values of $j$ and $\ell,$ we take the time derivative of $E_{j,\ell}:$
\begin{equation}\label{timeDerivativeEnergy}
\frac{dE_{j,\ell}}{dt}=\int_{D}\left(\partial_{a}^{j}\partial_{z}^{\ell}g^{n+1}\right)
\left(\partial_{a}^{j}\partial_{z}^{\ell}\partial_{t}g^{n+1}\right)\ dadz.
\end{equation}
We therefore need to write a helpful expression for $\partial_{a}^{j}\partial_{z}^{\ell}\partial_{t}g^{n+1}.$  
Applying derivatives to \eqref{gIteratedTransport}, we arrive at the expression
\begin{multline}\label{ellNonzero}
\partial_{a}^{j}\partial_{z}^{\ell}\partial_{t}g^{n+1}=\frac{1}{2}\sigma^{2}\partial_{a}^{j}\partial_{z}^{\ell+2}g^{n+1}
+(\ell+2)\sigma(\partial_{z}\sigma)\partial_{a}^{j}\partial_{z}^{\ell+1}g^{n+1}-\mu\partial_{a}^{j}\partial_{z}^{\ell+1}g^{n+1}
\\
-(\chi z+r^{n}\chi a)\partial_{a}^{j+1}\partial_{z}^{\ell}g^{n+1}
-\chi\Theta_{c}(y^{n},f^{n})\partial_{a}^{j+1}\partial_{z}^{\ell}g^{n+1}
+\Phi,
\end{multline}
where $\Phi$ is a collection of terms which will be more routine to estimate.  We can write $\Phi$ explicitly:
\begin{multline}\nonumber
\Phi=\frac{1}{2}\sum_{m=2}^{\ell}{\ell+2 \choose m}
\left(\partial_{z}^{m}\sigma^{2}\right)\partial_{a}^{j}\partial_{z}^{\ell+2-m}g^{n+1}
-\sum_{m=1}^{\ell}{\ell+1 \choose m}
\left(\partial_{z}^{m}\mu\right)\partial_{a}^{j}\partial_{z}^{\ell+1-m}g^{n+1}
\\
-\sum_{m=1}^{j}{j+1 \choose m}\left(\partial_{a}^{m}(\chi z+r^{n}\chi a)\right)\partial_{a}^{j+1-m}\partial_{z}^{\ell}g^{n+1}
\\
-\ell\sum_{m=0}^{j+1}{j+1 \choose m}(\partial_{a}^{m}\chi)
\partial_{a}^{j+1-m}\partial_{z}^{\ell-1}g^{n+1}\\
+\left[\partial_{a}^{j+1}\partial_{z}^{\ell}\left(g^{n+1}\chi\Theta_{c}(y^{n},f^{n})\right)
-\left(\partial_{a}^{j+1}\partial_{z}^{\ell}g^{n+1}\right)\chi\Theta_{c}(y^{n},f^{n})
\right].
\end{multline}
Using inequalities for Sobolev functions, we have an estimate for $\Phi,$ namely
\begin{equation}\nonumber
\|\Phi\|_{L^{2}}\leq c\left(1+|r^{n}(t)|+\|\Theta_{c}(y^{n},f^{n})\|_{H^{s+1}}\right)\|g^{n+1}\|_{H^{s}}.
\end{equation}
Since $\Theta_{c}$ is smooth and since the prior iterates satisfy \eqref{yInductiveHypothesis},
\eqref{rInductiveHypothesis}, and \eqref{fInductiveHypothesis}, we see that we may bound $\Phi$ by 
a constant (independent of our parameters $n$ and $\delta$) times the norm of $g^{n+1},$ i.e.
\begin{equation}\label{usefulPhiBound}
\|\Phi\|_{L^{2}}\leq c\|g^{n+1}\|_{H^{s}}.
\end{equation}

We proceed by substituting \eqref{ellNonzero} into \eqref{timeDerivativeEnergy}:
\begin{multline}\nonumber
\frac{dE_{j,\ell}}{dt}=\int_{D}\frac{\sigma^{2}}{2}\left(\partial_{a}^{j}\partial_{z}^{\ell}g^{n+1}\right)
\left(\partial_{a}^{j}\partial_{z}^{\ell+2}g^{n+1}\right)\ dadz\\
+
\int_{D}(\ell+2)\sigma(\partial_{z}\sigma)\left(\partial_{a}^{j}\partial_{z}^{\ell}g^{n+1}\right)
\left(\partial_{a}^{j}\partial_{z}^{\ell+1}g^{n+1}\right)\ dadz
\\
-\int_{D}\mu\left(\partial_{a}^{j}\partial_{z}^{\ell}g^{n+1}\right)
\left(\partial_{a}^{j}\partial_{z}^{\ell+1}g^{n+1}\right)\ dadz
\\
-\int_{D}(\chi z+r^{n}\chi a)\left(\partial_{a}^{j}\partial_{z}^{\ell}g^{n+1}\right)
\left(\partial_{a}^{j+1}\partial_{z}^{\ell}g^{n+1}\right)\ dadz
\\
-\int_{D}\chi \Theta_{c}(y^{n},f^{n})\left(\partial_{a}^{j}\partial_{z}^{\ell}g^{n+1}\right)
\left(\partial_{a}^{j+1}\partial_{z}^{\ell}g^{n+1}\right)\ dadz
+\int_{D}\Phi\left(\partial_{a}^{j}\partial_{z}^{\ell}g^{n+1}\right)\ dadz\\
=I+II+III+IV+V+VI.
\end{multline}

We integrate $I$ by parts with respect to $z$ and add the result to $II,$ finding
\begin{multline}\nonumber
I+II=-\int_{D}\frac{\sigma^{2}}{2}\left(\partial_{a}^{j}\partial_{z}^{\ell+1}g^{n+1}\right)^{2}\ dadz
\\
+(\ell+1)\int_{D}\sigma(\partial_{z}\sigma)\left(\partial_{a}^{j}\partial_{z}^{\ell}g^{n+1}\right)
\left(\partial_{a}^{j}\partial_{z}^{\ell+1}g^{n+1}\right)\ dadz,
\end{multline}
where the properties of $\sigma$ eliminate the presence of a boundary term.  The first integral on the right-hand
side could be used to find gain of regularity, but we will not need this for the present and we instead simply note
that it is nonpositive.  The second integral on the right-hand side can be integrated by parts with respect to $z$ once 
more (and there is again no boundary term), yielding
\begin{equation}\nonumber
I+II\leq -\frac{\ell+1}{2}\int_{D}\left(\sigma\partial_{z}^{2}\sigma+(\partial_{z}\sigma)^{2}\right)
\left(\partial_{a}^{j}\partial_{z}^{\ell}g^{n+1}\right)^{2}\ dadz.
\end{equation}
There exists $c>0,$ then, depending on the function $\sigma$ such that
\begin{equation}\label{firstTwoBounded}
I+II\leq cE.
\end{equation}

Next, we integrate $III$ by parts with respect to the $z$ variable, and we integrate each of $IV$ and $V$ by parts
with respect to the $a$ variable.  This yields the following:
\begin{equation}\nonumber
III=\int_{D}\frac{\partial_{z}\mu}{2}\left(\partial_{a}^{j}\partial_{z}^{\ell}g^{n+1}\right)^{2}\ dadz,
\end{equation}
\begin{equation}\nonumber
IV=\int_{D}\frac{\partial_{a}(\chi z+r^{n}\chi a)}{2}
\left(\partial_{a}^{j}\partial_{z}^{\ell}g^{n+1}\right)^{2}\ dadz,
\end{equation}
\begin{equation}\nonumber
V=\int_{D}\frac{\partial_{a}\left(\chi\Theta_{c}(y^{n},f^{n})\right)}{2}
\left(\partial_{a}^{j}\partial_{z}^{\ell}g^{n+1}\right)^{2}\ dadz.
\end{equation}
Here, there is no boundary term when integrating by parts in $III$ because of the properties of $\mu$ at
$z_{min}$ and $z_{max}.$  There are no boundary terms in $IV$ and $V$ when integrating by parts because of
the presence of the factors of $\chi.$
Just as we bounded $I+II$ in \eqref{firstTwoBounded}, we may bound $III:$
\begin{equation}\label{iiiBound}
III\leq cE.
\end{equation}
For $IV$ and $V,$ since they involve the prior iterates, we must utilize the inductive hypothesis.  For $IV,$ we use
\eqref{rInductiveHypothesis} to find
\begin{equation}\nonumber
IV\leq c\left(1+\frac{2P_{max}}{K_{data}}\right)E.
\end{equation}
Since the constants $P_{max}$ and $K_{data}$ are considered to be fixed (and especially, they do not depend on $n$
or $\delta$), 
we incorporate these into the constant $c$ to write this as 
\begin{equation}\label{ivBound}
IV\leq cE.
\end{equation}

Since the function $\Theta_{c}$ is continuous, there exists a constant $c>0$ such that for all $\tilde{y}$
and $\tilde{f}$ satisfying $\|\tilde{y}\|_{H^{s+1}}\leq M\|y_{T}\|_{H^{s+1}}$ and $\tilde{f}\in[1/2,2],$ we have
$\|\partial_{a}\Theta_{c}(\tilde{y},\tilde{f})\|_{L^{\infty}(D)}\leq c.$  In light of \eqref{yInductiveHypothesis} and
\eqref{fInductiveHypothesis}, then, we conclude
\begin{equation}\label{vBound}
V\leq c E.
\end{equation}

Finally, we may use \eqref{usefulPhiBound} directly to bound $VI$ as
\begin{equation}\label{viBound}
VI\leq cE.
\end{equation}
Adding \eqref{firstTwoBounded}, \eqref{iiiBound}, \eqref{ivBound}, \eqref{vBound}, and \eqref{viBound},
also summing over $j$ and $\ell,$ we have
\begin{equation}\nonumber
\frac{dE}{dt}\leq cE,
\end{equation}
with this constant $c$ independent of $n$ and $\delta.$

Thus, as claimed, for the given value $M>1$ chosen above, there exists $T_{g}>0$ such that if $T\in(0,T_{g})$,
then for all 
$t\in[0,T],$
\begin{equation}\nonumber
\|g^{n+1}(t,\cdot)\|_{H^{s}}\leq M\|g_{0}\|_{H^{s}},
\end{equation}
and this value of $T_{g}$ is independent of both our parameters $n$ and $\delta.$

The details for $y^{n+1}$ are very similar and we omit them.  Our conclusion is that there exists $T_{y}>0$ such
if $T\in(0,T_{y}),$ then for all $t\in[0,T],$
\begin{equation}\nonumber
\|y^{n+1}(t,\cdot)\|_{H^{s+1}}\leq M\|y_{T}\|_{H^{s+1}}.
\end{equation}
Again, this value of $T_{y}$ is independent of $n$ and $\delta.$

We now turn to the estimates for $r^{n+1},$ $f^{n+1},$ and $K^{n+1}.$  The bound for $r^{n+1}$
is immediate from the definition \eqref{rIterated}, the definition of $P_{max},$ and the bounds in the inductive hypothesis
\eqref{yInductiveHypothesis}, \eqref{gInductiveHypothesis}, \eqref{fInductiveHypothesis}, and 
\eqref{KInductiveHypothesis}.
Given the bound \eqref{rInductiveHypothesis} and the formula \eqref{fIteratedExplicit} for $f^{n+1},$ we see that there
exists $T_{f}>0,$ independent of $n$ and $\delta,$ such that if $T\in(0,T_{f})$ then for all $t\in[0,T]$ we have
$f^{n+1}(t)\in[1/2,2].$

We next deal with $K^{n+1},$ as defined in \eqref{KIteratedDefinition}.  Given the bounds on the $n$-th iterates
in the inductive hypothesis, we see that for sufficiently small values of the time horizon, $g^{n}$ remains 
close to the initial data $g_{0,\delta},$ 
$f^{n}$ remains close to its terminal value which is $f^{n}(T)=1,$ and $y^{n}$ remains
close to its terminal data $y_{T,\delta}.$  We conclude that there exists $T_{K}>0,$ with this value independent of $n$ and 
$\delta,$ such that if $T\in(0,T_{K}),$ then for all $t\in[0,T],$ we have $K^{n+1}(t)\geq K_{data}/2.$
(To be clear, we have already taken $\delta$ sufficiently small so that the initial iterate $K^{0}$ satisfies 
$K^{0}\geq 3K_{data}/4,$ and the value of $T_{K}$ is otherwise independent of $\delta.$)

Finally we wish to ensure that $y^{n+1}+f^{n+1}w_{\infty}$ remains bounded below by $W/2.$
Similarly to the bound for $K^{n+1},$ the bounds of the inductive hypothesis imply that the time derivatives of 
$y^{n+1}$ and
$f^{n+1}$ are uniformly bounded, and thus if $T$ is sufficiently small, the minimum of 
$y^{n+1}+f^{n+1}w_{\infty}$ remains close to its terminal value, which by \eqref{deltaData1} is at least $3W/4.$   
Thus there exists $T_{W}>0$ such that
if $T\in(0,T_{W}),$ then 
\begin{equation}\nonumber
\min_{(t,a,z)\in[0,T]\times D}\left(y^{n+1}+f^{n+1}w_{\infty}\right)\geq\frac{W}{2},
\end{equation}
and this value of $T_{W}$ is independent of $n$ and $\delta.$

Choosing $T_{*}=\min\{T_{g},T_{y},T_{f},T_{K},T_{W}\},$ the proof is complete.
\end{proof}

\section{Passage to the limit}\label{limitSection}

We now take the limit of our iterates, proving our main theorem.
\begin{theorem}\label{mainExistenceTheorem}
Let $s\in\mathbb{N}$ satisfying $s\geq 4$ be given, and let $w_{\infty}>0$ be given.  
Let $A>0$ be given and let the spatial domain $D$ be as above.
Let $y_{T}\in H^{s+1}(D)$ and $g_{0}\in H^{s}(D)$ be given, such that the support of $g_{0}$ with respect to $a$
and the support of $y_{T}$ with respect to $a$ are in the interior of the interval $[-A,A],$ and assume that $g_{0}$ is a 
probability measure.  Assume  $W>0,$ where $W$ is defined by \eqref{dataW}.
There exists $T_{**}>0$ such that if $T\in(0,T_{**}),$ then there exists 
$y\in L^{\infty}([0,T];H^{s+1}(D))\cap C([0,T];H^{s'+1}(D))$ for all $s'<s,$ and 
$g\in L^{\infty}([0,T];H^{s}(D))\cap C([0,T];H^{s'}(D))$ for all $s'<s,$ and
$f\in C^{1}([0,T]),$ such that $K[y,g,f]>0$ for all $t\in[0,T]$ and with $r$ defined by
$r=P[y,g,f]/K[y,g,f],$ then $(y,g,f)$ solve \eqref{yFinalEquation}, \eqref{gFinalEquation}, 
and \eqref{fFinalEquation} with data $g(0,\cdot)=g_{0},$ $y(T,\cdot)=y_{T},$ and 
$f(T)=1.$  The solution $g$ is a probability measure at each time $t\in[0,T]$ and
$y+fw_{\infty}$ is positive at each time $t\in[0,T].$
\end{theorem}

We make a remark on the data and the constraint $\mathcal{C}=0$ before beginning the proof.
\begin{remark}
Note that we have not required in our existence theorem that $\int\int ag_{0}\ dadz=0;$ the existence theorem holds
whether or not the constraint is initially satisfied.  Of course if one hopes to have $\mathcal{C}=0$ for all time, then the
initial data should be taken to satisfy $\mathcal{C}(t=0)=0.$
\end{remark}

\begin{proof}
We have previously suppressed the dependence of the solutions on the mollification parameter $\delta,$ and we have
left this value $\delta>0$ to be arbitrary.  We now consider the sequence of solutions resulting from taking a specific
value of $\delta$ for each $n\in\mathbb{N},$ namely $\delta=1/n.$  In this section we will show that
there is a subsequence of $(y^{n},g^{n},f^{n},K^{n},r^{n})$ which converges to a solution of the transformed system.

We restrict $T$ to values in $(0,T_{*}),$ with $T_{*}$ given by Theorem \ref{uniformBoundTheorem}.
There will be another restriction on $T$ later in the proof.

We begin with $y^{n}.$  By Theorem \ref{uniformBoundTheorem}, on the time interval $[0,T],$ the sequence
$y^{n}$ is uniformly bounded in $H^{s+1}(D).$  With $s\geq 2,$ Sobolev embedding then implies that $\nabla_{a,z}y^{n}$ 
is bounded in $L^{\infty}$ uniformly with respect to $n.$   Insepecting the family of evolution equations
\eqref{yIteratedTransport}, again using the uniform bounds of Theorem \ref{uniformBoundTheorem} and now
using $s\geq 3,$ we see that $\partial_{t}y^{n}$ is bounded in $L^{\infty},$ uniformly with respect to $n.$
We conclude that $\{y^{n}:n\in\mathbb{N}\}$ is an equicontinuous family, and we apply the Arzela-Ascoli theorem
to find a uniformly convergent subsequence, which we do not relabel.  We call the limit $y.$

We now address regularity of the limit.  The Arzela-Ascoli theorem gives 
convergence in $C([0,T]\times D),$ and this immediately implies convergence in $C([0,T];L^{2}(D)).$
With the uniform bound in $H^{s+1}$ from Theorem \ref{uniformBoundTheorem}, Sobolev interpolation then 
implies convergence in $C([0,T];H^{s'+1}),$ for any $s'<s.$  Furthermore, since the iterates are uniformly bounded
in $H^{s+1},$ at every time $t\in[0,T]$ there is a weak limit in $H^{s+1}$ obeying the same bound, and this limit 
must again equal $y.$  Thus $y$ is also in $L^{\infty}([0,T];H^{s+1}).$

The argument for $g^{n}$ is the same, except that $g^{n}$ being bounded in $H^{s}$ rather than $H^{s+1}$ means
that we require $s\geq 4$ to have the $L^{\infty}$ bound on the time derivatives.  We call the limit of the subsequence
(which we do not relabel) $g.$  This $g$ is in $C([0,T];H^{s'})$ for any $s'<s,$ and also is in $L^{\infty}([0,T];H^{s}).$

We next take the limit of $f^{n}.$  From the uniform bounds of Theorem \ref{uniformBoundTheorem}, inspection
of \eqref{fIteratedEquation} implies that $(f^{n})'$ is uniformly bounded.  Thus $\{f^{n}:n\in\mathbb{N}\}$ is an 
equicontinuous family on $[0,T].$  The Arzela-Ascoli theorem again applies, yielding a uniform limit of a subsequence
(which we do not relabel); we call the limit $f.$

We turn now to the sequence $K^{n}.$  Considering \eqref{KIteratedDefinition} and the fact that $y^{n},$ $g^{n},$
and $f^{n}$ all converge uniformly, we see that $K^{n}$ converges to a limit $K$ which is given by
\begin{equation}\nonumber
K=\int\int g(H_{pp}(\psi(y+fw_{\infty})))(y+fw_{\infty})\ dadz,
\end{equation}
and this convergence is uniform.
Since we have $K^{n}\geq K_{data}/2$ for all $n,$ we also have $K\geq K_{data}/2$ for all times.

Finally we consider convergence of $r^{n}.$  In light of \eqref{rIterated} and since we know that $K^{n}$ converges,
we only need to consider the convergence of $P^{n}.$  The convergence that we have established for $y^{n}$
implies that up through second derivatives of $y^{n}$ converge to the appropriate derivatives of $y.$ 
Similarly, up to second derivatives of $g^{n}$ converge uniformly to the 
appropriate derivatives of $g.$  This is enough regularity to ensure that $P^{n}=P_{c}[y^{n},f^{n},g^{n}]$ converges
uniformly to $P_{c}[y,f,g].$  Since $P^{n}$ and $K^{n}$ both converge, we see that $r^{n}$ converges to
$r=P/K,$ and this convergence is uniform.

We next demonstrate that the limits $y$ and $g$ satisfy the appropriate equations.  We provide the details only for
$g,$ as the argument for $y$ is the same.  We integrate \eqref{gIteratedTransport} with respect to time, on the interval
$[0,t]:$
\begin{multline}\nonumber
g^{n+1}(t,\cdot)=g_{0,1/n}+\int_{0}^{t}
\frac{1}{2}\partial_{zz}(\sigma^{2}g^{n+1})-\partial_{z}(\mu g^{n+1})
-\partial_{a}(\chi(z+r^{n}a)g^{n+1})
\\
-\partial_{a}(\chi\Theta_{c}(y^{n},f^{n})g^{n+1})\ dt'.
\end{multline}
The uniform convergence of the iterates $y^{n},$ $g^{n},$ $f^{n},$ and $r^{n}$ implies convergence of the integral. 
Taking the limit, we have
\begin{equation}\nonumber
g(t,\cdot)=g_{0}+\int_{0}^{t}
\frac{1}{2}\partial_{zz}(\sigma^{2}g)-\partial_{z}(\mu g)
-\partial_{a}(\chi(z+ra)g)
-\partial_{a}(\chi\Theta_{c}(y,f)g)\ dt'.
\end{equation}
Taking the time derivative of this, we see that
\begin{equation}\label{gLimitCutoffs}
\partial_{t}g-\frac{1}{2}\partial_{zz}(\sigma^{2}g)+\partial_{z}(\mu g)+\partial_{a}(\chi(z+ra)g)
+\partial_{a}(\chi\Theta_{c}(y,f)g)=0.
\end{equation}
Similarly, we conclude that the equation satisfied by $y$ is
\begin{equation}\label{yLimitCutoffs}
\partial_{t}y+\frac{1}{2}\sigma^{2}y+\mu\partial_{z}y+(r-\rho)y
+(\chi(z+ra))\partial_{a}y+\chi\Theta_{c}(y,f)\partial_{a}y=0.
\end{equation}

The last step in the existence proof is to remove the cutoff functions $\chi$ and $\psi.$
As discussed when the iterative scheme was set up, as long as the iterates remain uniformly bounded, and
as long as $T$ is small enough, the support of the iterates with respect to the $a$ variable 
remains within the set $A_{2}.$  Since the iterates
converge uniformly, the support of $y$ and $g$ with respect to the $a$ variable 
also remains confined to the set $A_{2}$ throughout the interval
$[0,T].$  Since the cutoff function satisfies $\chi=1$ when restricted to $A_{2}$, we see that in \eqref{gLimitCutoffs} and
\eqref{yLimitCutoffs}, we have $\chi g=g$ and $\chi\partial_{a}y=\partial_{a}y.$
Since \eqref{WInductiveHypothesis} is satisfied for all $n$ and since $y^{n}$ and $f^{n}$ converge to $y$ and $f,$
we see that
\begin{equation}\nonumber
\min_{(t,a,z)\in[0,T]\times D}\left(y(t,a,z)+f(t)w_{\infty}\right)\geq\frac{W}{2}.
\end{equation}
This implies $\psi(y+fw_{\infty})=y+fw_{\infty},$ and therefore that $\Theta_{c}(y,f)=\Theta(y,f).$
We conclude that the equations satisfied by $y$ and $g$ are \eqref{yFinalEquation} and
\eqref{gFinalEquation}, as desired.

The proof of the existence theorem is complete.
\end{proof}

\section{Uniqueness}\label{uniquenessSection}

We now prove uniqueness of our solutions.

\begin{theorem}\label{mainUniquenessTheorem} 
Let $s\in\mathbb{N}$ satisfying $s\geq 4$ be given, and let $w_{\infty}>0$ be given.  
Let $A>0$ be given and let the spatial domain $D$ be as above.
Let $y_{T}\in H^{s+1}(D)$ and $g_{0}\in H^{s}(D)$ be given, such that the support of $g_{0}$ with respect to $a$
and the support of $y_{T}$ with respect to $a$ are in the interior of the interval $[-A,A],$ and assume that $g_{0}$ is a 
probability measure.  Assume $W>0,$ where $W$ is defined by \eqref{dataW}.
Let $(y_{1},g_{1},f_{1})$ and $(y_{2},g_{2},f_{2})$ and the associated interest rates
$r_{i}=P[y_{i},g_{i},f_{i}]/K[y_{i},g_{i},f_{i}]$ satisfy
\eqref{yFinalEquation}, \eqref{gFinalEquation}, \eqref{fFinalEquation}, with data 
$g_{i}(0,\cdot)=g_{0},$ $y_{i}(T,\cdot)=y_{T},$ and $f_{i}(T)=1.$
Let $T>0$ be such that $y_{i}\in L^{\infty}([0,T];H^{s+1}(D))\cap C([0,T];H^{s'+1}(D)),$ for all $s'<s,$
and such that $g_{i}\in L^{\infty}([0,T];H^{s}(D))\cap C([0,T];H^{s'}(D)),$ for all $s'<s,$ and such
that $f_{i}\in C^{1}([0,T]).$  Assume that $g_{i}$ and $y_{i}$ are compactly supported with respect to the $a$
variable in the interval $(-2A,2A).$
There exists $T_{***}$ such that if $T\in(0,T_{***}),$ then $(y_{1},g_{1},f_{1})=(y_{2},g_{2},f_{2}).$
\end{theorem}

\begin{proof}
By arguments such as those in \cite{constantinEscher} we see that the evolution for $\partial_{a}v=y+fw_{\infty}$
is positivity preserving, and thus we do not need to assume an explicit lower bound for $y+fw_{\infty}$ over the
given interval $[0,T].$  Similarly we could dispense with the explicit bound on the support with respect to the $a$ 
variable, but we do state it here so as to keep the domain consistent with the solutions we have already proved
to exist.

We define three components for the energy for the difference of two solutions, called
$E_{d,g},$ $E_{d,y},$ and $E_{d,f},$ where
\begin{equation}\nonumber
E_{d,g}(t)=\frac{1}{2}\int\int(g_{1}-g_{2})^{2}\ dadz,
\end{equation}
\begin{equation}\nonumber
E_{d,y}(t)=\frac{1}{2}\int\int|\nabla_{a,z}(y_{1}-y_{2})|^{2}\ dadz,
\end{equation}
and
\begin{equation}\nonumber
E_{d,f}=\sup_{t\in[0,T]}\frac{1}{2}|f_{1}(t)-f_{2}(t)|^{2}.
\end{equation}
Note that $E_{d,g}(0)=0$ and that $E_{d,y}(T)=0.$

We start by estimating $E_{d,f}.$  Noting that for $i\in\{1,2\}$ we have the equations 
$f_{i}'=(\rho-r_{i})f_{i},$ and that $f_{i}(T)=1,$ we can write, for any $t\in[0,T],$
\begin{equation}\nonumber
\frac{1}{2}|f_{1}(t)-f_{2}(t)|^{2}=-\int_{t}^{T}(f_{1}(t')-f_{2}(t'))(f_{1}'(t')-f_{2}'(t'))\ dt'.
\end{equation}
Substituting from the equations for $f_{i}'$ and adding and subtracting, this becomes
\begin{equation}\nonumber
\frac{1}{2}|f_{1}(t)-f_{2}(t)|=-\int_{t}^{T}(f_{1}-f_{2})(r_{2}-r_{1})f_{1}\ dt'
-\int_{t}^{T}(f_{1}-f_{2})^{2}(\rho-r_{2})\ dt'.
\end{equation}
Taking the supremum with respect to time and performing some other manipulations, we can bound this as
\begin{equation}\label{EdfIntermediate}
E_{d,f}\leq cTE_{d,f}+cT|r_{1}-r_{2}|_{L^{\infty}}^{2}.
\end{equation}

We will now work with $r_{1}-r_{2}.$  Since $r_{i}=P_{i}/K_{i},$ it is clear that at any time $r_{1}-r_{2}$ can be bounded
in terms of $K_{1}-K_{2}$ and $P_{1}-P_{2}.$  We consider $K_{1}-K_{2}$ first.  For any $t\in[0,T],$ we have
\begin{multline}\nonumber
|K_{1}(t)-K_{2}(t)|=\Bigg|\int\int g_{1}(H_{pp}(y_{1}+f_{1}w_{\infty}))(y_{1}+f_{1}w_{\infty})\ dadz
\\
-\int\int g_{2}(H_{pp}(y_{2}+f_{2}w_{\infty})(y_{2}+f_{2}w_{\infty})\ dadz\Bigg|.
\end{multline}
After some adding and subtracting and using a Lipschitz estimate for $H_{pp},$ it is evident that this can be bounded by
\begin{equation}\label{KDifferenceAtT}
|K_{1}(t)-K_{2}(t)|\leq c(E_{g}^{1/2}+E_{y}^{1/2}+E_{f}^{1/2}).
\end{equation}
We will need \eqref{KDifferenceAtT} as well as the supremum of this with respect to time,
\begin{equation}\label{KDifferenceSup}
|K_{1}-K_{2}|_{L^{\infty}}\leq c\left(\sup_{t\in[0,T]}\left(E_{g}^{1/2}+E_{y}^{1/2}\right)+E_{f}^{1/2}\right).
\end{equation}

The difference $P_{1}-P_{2}$ is similar but slightly more involved, as we must integrate by parts in some instances.
We start by noting that the definition \eqref{gFinalEquation} of $P$ includes three terms, so we decompose $P_{1}-P_{2}$
as
\begin{equation}\nonumber
P_{1}-P_{2}=\Upsilon_{I}+\Upsilon_{II}+\Upsilon_{III}.
\end{equation}
We believe the meaning here is clear, and we will only write out $\Upsilon_{II}$ in detail.
We add and subtract to decompose $\Upsilon_{II}$ as
\begin{equation}\nonumber
\Upsilon_{II}=\sum_{j=1}^{7}\Upsilon_{II,j},
\end{equation}
where we have the following definitions:
\begin{multline}\nonumber
\Upsilon_{II,1}=\int\int(g_{1}-g_{2})H_{pp}(y_{1}+f_{1}w_{\infty})\cdot
\\
\cdot\left(-\frac{1}{2}\sigma^{2}\partial_{zz}y_{1}-\mu\partial_{z}y_{1}-H_{p}(y_{1}+f_{1}w_{\infty})\partial_{a}y_{1}
+\rho(y_{1}+f_{1}w_{\infty})\right)\ dadz,
\end{multline}
\begin{multline}\nonumber
\Upsilon_{II,2}=\int\int g_{2}\left[H_{pp}(y_{1}+f_{1}w_{\infty})-H_{pp}(y_{2}+f_{2}w_{\infty})\right]\cdot
\\
\cdot\left(-\frac{1}{2}\sigma^{2}\partial_{zz}y_{1}-\mu\partial_{z}y_{1}-H_{p}(y_{1}+f_{1}w_{\infty})\partial_{a}y_{1}
+\rho(y_{1}+f_{1}w_{\infty})\right)\ dadz,
\end{multline}
\begin{equation}\nonumber
\Upsilon_{II,3}=\int\int g_{2}H_{pp}(y_{2}+f_{2}w_{\infty})
\left(-\frac{1}{2}\sigma^{2}\partial_{zz}(y_{1}-y_{2})\right)\ dadz,
\end{equation}
\begin{equation}\nonumber
\Upsilon_{II,4}=\int\int g_{2}H_{pp}(y_{2}+f_{2}w_{\infty})
\left(-\mu\partial_{z}(y_{1}-y_{2})\right)\ dadz,
\end{equation}
\begin{equation}\nonumber
\Upsilon_{II,5}=\int\int g_{2}H_{pp}(y_{2}+f_{2}w_{\infty})
\left(-H_{p}(y_{1}+f_{1}w_{\infty})+H_{p}(y_{2}+f_{2}w_{\infty})\right)\partial_{a}y_{1}\ dadz,
\end{equation}
\begin{equation}\nonumber
\Upsilon_{II,6}=\int\int g_{2}H_{pp}(y_{2}+f_{2}w_{\infty})
\left(-H_{p}(y_{2}+f_{2}w_{\infty})\partial_{a}(y_{1}-y_{2})\right)\ dadz,
\end{equation}
and
\begin{equation}\nonumber
\Upsilon_{II,7}=\int\int g_{2}H_{pp}(y_{2}+f_{2}w_{\infty})\rho(y_{1}-y_{2}+(f_{1}-f_{2})w_{\infty})\ dadz.
\end{equation}
It is immediate that $\Upsilon_{II,1},$ $\Upsilon_{II,2},$ $\Upsilon_{II,5},$ and $\Upsilon_{II,7}$ may be bounded
in terms of $E_{d,g},$ $E_{d,y},$ and $E_{d,f};$ note that Lipschitz estimates for $H_{p}$ and $H_{pp}$ are needed,
but these are smooth functions and the Lipschitz estimates are thus available.  The terms $\Upsilon_{II,3}$ and
$\Upsilon_{II,4}$ may be bounded in terms of $E_{d,y}$ after integrating by parts with respect to the $z$ variable;
note that there are no boundary terms because of the properties of the diffusion and transport coefficients, $\sigma$ 
and $\mu,$ at the boundaries of the domain.  The term $\Upsilon_{II,6}$ can be bounded in terms
of $E_{d,y}$ after an integration by parts with respect to the $a$ variable; there is no boundary term because
solutions are compactly supported with respect to the $a$ variable.
We omit further details, and the result of these and similar considerations
is the bounds
\begin{equation}\nonumber
|P_{1}(t)-P_{2}(t)|\leq c(E_{g}^{1/2}+E_{y}^{1/2}+E_{f}^{1/2}),
\end{equation}
and 
\begin{equation}\nonumber
|P_{1}-P_{2}|_{L^{\infty}}\leq c\left(\sup_{t\in[0,T]}\left(E_{g}^{1/2}+E_{y}^{1/2}\right)+E_{f}^{1/2}\right).
\end{equation}

From our bounds on differences of $K$ and $P,$ we conclude that at each $t\in[0,T],$
\begin{equation}\nonumber
|r_{1}(t)-r_{2}(t)|\leq c\left(E_{g}^{1/2}+E_{y}^{1/2}+E_{f}^{1/2}\right),
\end{equation}
and taking the supremum in time,
\begin{equation}\label{rDifferenceBound}
|r_{1}(t)-r_{2}(t)|_{L^{\infty}}\leq c\left(\sup_{t\in[0,T]}\left(E_{g}^{1/2}+E_{y}^{1/2}\right)+E_{f}^{1/2}\right).
\end{equation}
Using this in \eqref{EdfIntermediate}, our bound for $E_{d,f}$ is
\begin{equation}\label{EdfFinal}
E_{d,f}\leq cT\left(\sup_{t\in[0,T]}\left(E_{d,g}+E_{d,y}\right)+E_{d,f}\right).
\end{equation}

We next establish that there exists $c>0$ such that
\begin{equation}\label{uniqToIntegrate1}
\frac{dE_{d,g}}{dt}\leq c(E_{d,g}+E_{d,y}+E_{d,f}),
\end{equation}
and
\begin{equation}\label{uniqToIntegrate2}
\frac{dE_{d,y}}{dt}\leq c(E_{d,g}+E_{d,y}+E_{d,f}).
\end{equation}
To this end, we take the time derivative of $E_{d,g}:$
\begin{equation}\nonumber
\frac{dE_{d,g}}{dt}=\int\int(g_{1}-g_{2})\partial_{t}(g_{1}-g_{2})\ dadz.
\end{equation}
We then substitute from the equations \eqref{gFinalEquation} satisfied by each $g_{i},$ and
add and subtract:
\begin{multline}\nonumber
\frac{dE_{d,g}}{dt}=\frac{1}{2}\int\int(g_{1}-g_{2})\partial_{zz}(\sigma^{2}(g_{1}-g_{2}))\ dadz
\\
-\int\int(g_{1}-g_{2})\partial_{z}(\mu(g_{1}-g_{2}))\ dadz
-\int\int(g_{1}-g_{2})\partial_{a}((r_{1}-r_{2})ag_{1})\ dadz
\\
-\int\int(g_{1}-g_{2})\partial_{a}((z+r_{2}a)(g_{1}-g_{2}))\ dadz
\\
-\int\int(g_{1}-g_{2})\partial_{a}((g_{1}-g_{2})\Theta(y_{1},f_{1}))\ dadz
\\
-\int\int(g_{1}-g_{2})\partial_{a}(g_{2}(\Theta(y_{1},f_{1})-\Theta(y_{2},f_{2})))\ dadz.
\end{multline}
There are six terms on the right-hand side, and estimating these is very much like the estimate for $g^{n+1}$ in
the proof of Theorem \ref{uniformBoundTheorem}.  Specifically, for the first term, the two derivatives with respect
to $z$ should be applied, and then some integrations by parts can be made.  For the second term, the one
derivative with respect to $z$ can be applied, and then an integration by parts can be made.  To estimate the third term, 
the bound \eqref{rDifferenceBound} is employed.  For the fourth and fifth terms, the derivative with respect to $a$ 
should be applied, and then an integration by parts can be made.  For the sixth term, the derivative with respect
to $a$ should be applied, a Lipschitz estimate for $\Theta$ is used, and a further addition and subtraction can be 
utilized as well.  We omit further details.

Integrating \eqref{uniqToIntegrate1} forward in time, and using the initial data, we find
\begin{multline}\label{uniqToAdd1}
E_{d,g}(t)\leq c\int_{0}^{t}E_{d,g}(t')+E_{d,y}(t')+E_{d,f}\ dt'
\\
\leq cT\left(\sup_{t\in[0,T]}\left(E_{d,g}(t)+E_{d,y}(t)\right)+E_{d,f}\right).
\end{multline}
Integrating \eqref{uniqToIntegrate2} backward in time, and using the terminal data, we find
\begin{multline}\label{uniqToAdd2}
E_{d,y}(t)\leq c\int_{t}^{T}E_{d,g}(t')+E_{d,y}(t')+E_{d,f}\ dt'
\\
\leq cT\left(\sup_{t\in[0,T]}\left(E_{d,g}(t)+E_{d,y}(t)\right)+E_{d,f}\right).
\end{multline}
Adding \eqref{EdfFinal}, 
\eqref{uniqToAdd1}, and \eqref{uniqToAdd2}, and taking the supremum in time and reorganizing terms,
we find
\begin{equation}\nonumber
(1-cT)\left(E_{d,f}+\sup_{t\in[0,T]}\left(E_{d,g}(t)+E_{d,y}(t)\right)\right)\leq 0.
\end{equation}
Thus as long as $0<T<1/c,$ we have $y_{1}=y_{2},$  $g_{1}=g_{2},$ and $f_{1}=f_{2}.$
\end{proof}

\section{Discussion}\label{discussionSection}

We mentioned above that we would remark again on the difference between choice of spatial domain here 
as compared to \cite{mollPhilTrans}.  We have taken the same domain with respect to the $z$ variable, but 
in \cite{mollPhilTrans} the domain with respect to the $a$ variable was taken to be $[a_{min},\infty)$ for a given
$a_{min}<0.$  We have instead taken the initial support of our functions with respect to the $a$ variable
in $[-A,A]$ for a given $A>0$ and the support of our solutions has remained in $[-2A,2A]$ over the time interval
$[0,T].$  If we take the view that $a_{min}<-2A,$ then our solutions fit into the framework of \cite{mollPhilTrans}
with regard to this aspect.

We also mentioned above that we would comment on our assumption that the range of $u'$ is equal to 
$(0,\infty),$ as would be the case, for instance, if $u(c)=\sqrt{c}.$  This assumption is only for simplicity
and the general case can be treated by our same method.  We stated in Section \ref{formulationSection}
that the quantity in the definition of the Hamiltonian is maximized when $p=u'(c),$ and so $c=(u')^{-1}(p).$
This formula is still valid if $p$ is in $\mathrm{Range}(u'),$ which is necessarily an interval.  Thus the formula
we have used throughout the work is valid for values of $p$ in a given interval.  But $p$ stands in for $\partial_{a}v,$
and the method we have applied does find solutions where $\partial_{a}v=y+fw_{\infty}$ only takes values in a given
interval.  For a general utility function, the terminal data $y_{T}+w_{\infty}$ can be taken
with values in the appropriate interval, and the time horizon $T$ can be taken sufficiently small so that solutions
remain in this interval.

We have discussed in the introduction that we can show that in some cases, the solutions we have proved to exist
are not solutions of the original system.  For every solution we have proved to exist via Theorem 
\ref{mainExistenceTheorem}, there is
associated a value of the constant $\mathcal{Q}.$  If this constant $\mathcal{Q}$ is nonzero, 
then the solution does not solve the original problem, i.e., if $\mathcal{Q}\neq0,$ then $\mathcal{C}\not\equiv0.$
It is straightforward to see that we can guarantee in some cases that $\mathcal{Q}\neq0.$  We define
\begin{equation}\nonumber
\mathcal{Q}_{data}=
\int\int (z+H_{p}(y_{T}+w_{\infty}))g_{0}\ dadz.
\end{equation}
Assume that $g_{0},$ $y_{T},$ and $w_{\infty}$ are specified such that $\mathcal{Q}_{data}\neq0.$  
Then for sufficiently small
values of $T>0,$ for the solutions $(y,g,f)$ proved to exist in our main theorem, the solution will not vary much 
from the data.  Therefore for small values of $T,$ we will have $\mathcal{Q}$ close to 
$\mathcal{Q}_{data},$ and $\mathcal{Q}$ will therefore be
nonzero.  

As mentioned in the introduction, the authors of \cite{mollPhilTrans} proposed a restriction on the choice of 
terminal values for $v,$ and thus in our case for the terminal data for $\partial_{a}v,$ which is
$y_{T}+w_{\infty}.$  In particular they proposed that $T$ should
be taken to be fairly large and the final value of $v$ should be associated to a stationary solution.  Since a stationary
solution can be viewed as the infinite-$T$ limit of solutions of the system under consideration, and stationary
solutions would satisfy $\mathcal{C}=0,$ this proposed data would be expected to yield solutions satisfying only
$\mathcal{C}\approx 0.$  Further work is warranted, though, to find solutions which satisfy the constraint exactly.
Specifically, given a value of the time horizon $T$ and the initial distribution $g_{0}$ initially satisfying the constraint, 
the author intends to 
perform computational and analytical studies seeking existence of
terminal data $y_{T}+w_{\infty}$ which yield $\mathcal{Q}=\mathcal{C}=0.$

\section*{Acknowledgments}
The author gratefully acknowledges support from the National Science Foundation through 
grant  DMS-1515849.  

\section*{Conflict of Interest}
Conflict of Interest: The authors declare that they have no conflict of interest.

\bibliography{ambroseMeanField}{}
\bibliographystyle{plain}

\end{document}